\documentclass[11pt]{amsart}
\usepackage{graphicx} 
\usepackage{amsfonts}
\usepackage{amsmath}
\usepackage{amssymb}
\usepackage{amsthm}
\usepackage{bm}
\usepackage{graphicx} 
\usepackage{comment}
\usepackage{textcomp}

\usepackage{amsfonts,latexsym,amscd, mathtools,tikz,circuitikz,hyperref}
\usepackage{cleveref}
\usepackage[dvipsnames]{xcolor}

\newcommand{\La}{\Lambda}
\newcommand{\Gr}{\text{Gr}}
\newcommand{\R}{\mathbb{R}}
\newcommand{\Z}{\mathbb{Z}}
\newcommand{\C}{\mathbb{C}}
\newcommand{\st}{\text{st}}
\newcommand{\Symp}{\text{Symp}}
\newcommand{\w}{\mathfrak{w}}
\newcommand{\D}{\mathbb{D}}

\newcommand{\FM}{\mathfrak{M}}
\renewcommand{\S}{\mathbb{S}}

\theoremstyle{plain}
\newtheorem{theorem}{Theorem}[section]
\newtheorem{lemma}[theorem]{Lemma}
\newtheorem{proposition}[theorem]{Proposition}
\newtheorem{corollary}[theorem]{Corollary}
\newtheorem{conjecture}[theorem]{Conjecture}

\theoremstyle{definition}
\newtheorem{definition}[theorem]{Definition}
\newtheorem{remark}[theorem]{Remark} 
\newtheorem{ex}[theorem]{Example}

\crefname{theorem}{theorem}{theorems}
\Crefname{theorem}{Theorem}{Theorems}
\crefname{definition}{definition}{definitions}
\Crefname{definition}{Definition}{Definitions}
\crefname{example}{example}{examples}
\Crefname{example}{Example}{Examples}
\crefname{lemma}{lemma}{lemmas}
\Crefname{lemma}{Lemma}{Lemmas}
\crefname{proposition}{proposition}{propositions}
\Crefname{proposition}{Proposition}{Propositions}
\crefname{corollary}{corollary}{corollaries}
\Crefname{corollary}{Corollary}{Corollaries}
\crefname{remark}{remark}{remarks}
\Crefname{remark}{Remark}{Remarks}

\title{Exact Lagrangian fillings of twist-spun torus links}
\author{Vincent Chen}
\author{Patton Galloway}
\author{James Hughes}
\author{Luciana Wei}

\begin{document}
\begin{abstract} 

We construct exact Lagrangian fillings of Legendrian torus links $\La(k, n-k)$ that are fixed by a Legendrian loop that acts by $2\pi\ell/n$ rotation. Using these rotationally symmetric fillings, we produce fillings of the corresponding Legendrian twist-spun tori. Our construction is combinatorial in nature, relating symmetric weakly separated collections and plabic graphs to symmetric Legendrian weaves via the T-shift procedure of Casals, Le, Sherman-Bennett, and Weng. The main technical ingredient in this process is a necessary and sufficient condition for the existence of maximal weakly separated collections of $k$-element subsets of $\{1, \dots, n\}$ that are fixed by addition of $\ell$ modulo $n$.

\end{abstract}
\maketitle

\section{Introduction}

\subsection{Context}
The classification of exact Lagrangian fillings of Legendrian links is an important problem in low-dimensional contact and symplectic topology. While the unknot and the Hopf link are the only two Legendrians for which there exist a complete classification of fillings \cite{EliashbergPolterovich96, Thomson25}, the last decade has seen a significant advancement in our understanding of constructing and distinguishing fillings \cite{EHK, YuPan, TreumannZaslow, CasalsNg, capovillasearle2023newton}. Many of the most recent works related to this problem, including \cite{STWZ, CasalsGao, GSW, ABL22, Hughes2021, CasalsWeng}, are closely related to the theory of cluster algebras. In particular, \cite[Conjecture 5.1]{CasalsLagSkel} describes a conjectural ADE-type classification, giving a 1-1 correspondence between exact Lagrangian fillings of braid positive Legendrians and cluster seeds of corresponding cluster algebras. This is accompanied by \cite[Conjecture 5.4]{CasalsLagSkel}, which posits a BCFG-type classification of fillings of Legendrian links admitting certain symmetries.

In \cite{HughesRoy}, the third author and Agniva Roy expanded this conjectural correspondence between symmetric fillings and cluster seeds to include exact Lagrangian fillings of a particular family of Legendrian surfaces known as Legendrian twist-spun tori. These twist-spun tori are first studied in \cite{EK}, and they are constructed as the mapping tori of a Legendrian loop of a Legendrian link, i.e., a Legendrian isotopy starting and ending at the same front diagram. While Legendrian loops appear in other contexts---most notably the first construction of infinite families of exact Lagrangian fillings in \cite{CasalsGao}---the twist-spun tori constructed from them give an alternative interpretation for \cite[Conjecture 5.4]{CasalsLagSkel} through an adaptation of the cluster structures of \cite{CasalsWeng} to this higher dimensional context. 

From a cluster-theoretic perspective, twist-spinning corresponds to a folding operation on a cluster algebra analogous to folding Dynkin diagrams in classical Lie theory. Beyond the finite-type case, one particularly rich source of foldings arises from the action of the cyclic shift automorphism $\rho$ on the Grassmannian $\Gr(k, n)$, studied by Chris Fraser in \cite{Fraser2020}; see also \cite{KaufmanSpecialFolding}. In loc.\ cit., Fraser gives a conjectural generalized cluster structure on (the distinguished component of) the cyclic symmetry locus $\Gr(k, n)^{\rho^\ell}$, which the third author used to distinguish a collection of fillings of twist-spuns of $(2, n)$-torus links in \cite{HughesRoy}.

\subsection{Main results}
In this paper, we construct exact Lagrangian fillings of Legendrian torus links exhibiting certain rotational symmetry, and from these, fillings of corresponding twist-spun tori. Let $\La(k, n-k)$ denote the Legendrian $(k, n-k)$ torus link depicted in Figure~\ref{fig: -1Closure}, where we require $n-k\geq k$. We can describe $\La(k, n-k)$ as  the $(-1)$-closure of the positive braid $\beta=(\sigma_1\dots \sigma_{k-1})^{n}$ where $\sigma_i$ denotes the $i$th Artin generator of the braid group.
 By \cite{Kalman2005}, $\La(k, n-k)$ admits a Legendrian loop $\psi$ that can be described as conjugation of $\beta$ by the Coxeter element $\sigma_1\dots\sigma_{k-1}$. Our main result gives a sufficient condition on $k, n,$ and $\ell$ for the existence of a filling of the Legendrian twist spun $\Sigma_{\psi^\ell}(\La(k, n-k))$,  constructed as the mapping torus of the Legendrian loop $\psi^\ell$. Let us fix the notation $d=n/\gcd(n, \ell)$.

\begin{theorem}\label{thm:intro_twist-spun_fillings}
If $k$ is congruent to $-1, 0,$ or $1$ modulo $d$, then $\Sigma_{\psi^\ell}(\La(k, n-k))$ admits an orientable exact Lagrangian filling.    
\end{theorem}

Our proof of Theorem \ref{thm:intro_twist-spun_fillings} is largely combinatorial in nature. We present our fillings of $\Sigma_{\psi^\ell}(\La(k, n-k))$ as mapping tori of Legendrian weave fillings of $\La(k, n-k)$ with $2\pi\ell/n$ rotational symmetry. Legendrian weaves are Legendrian surfaces whose front projections have a restricted set of singularities that can be encoded as a edge-colored graph. In the context of positroid varieties, these Legendrian weaves can be obtained by applying a process developed in \cite{CLSW2023} known as T-shift. This process inputs a reduced plabic graph, and outputs a Legendrian weave that induces the same cluster seed as the original plabic graph. 
\begin{center}
		\begin{figure}[h!]{ \includegraphics[width=.35\textwidth]{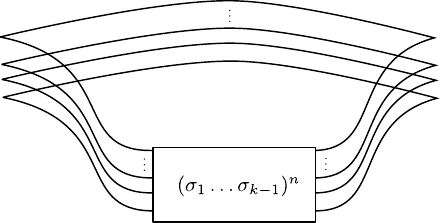}}\caption{A front projection of $\La(k, n-k)$ given as the $-1$-closure of the braid $(\sigma_1\dots\sigma_{k-1})^n$.}
			\label{fig: -1Closure}\end{figure}
	\end{center}

After verifying that the T-shift process preserves rotational symmetry, we prove Theorem~\ref{thm:intro_twist-spun_fillings} by producing rotationally symmetric plabic graphs, i.e., planar bicolored graphs that commonly appear as tools for encoding positroid combinatorics. This problem reduces---via the dual construction of plabic tilings---to finding maximal collections of weakly separated subsets of $[n]:=\{1, \dots, n\}$ of size $k$ that are fixed by the addition of $\ell$ modulo $n$; see Section~\ref{sec: combinatorics_background} for the definition of weak separation. 

\begin{theorem}\label{thm: intro_WSCs}
Given $k, n,$ and $\ell$, there is a $\rho^\ell$-symmetric weakly separated collection $D\subseteq {[n] \choose k}$ of size $k(n-k)+1$ if and only if $k$ is congruent to $-1, 0,$ or $1$ modulo $d$. 
\end{theorem}

Here $\rho^\ell$-symmetric means that the set $D$ is fixed under the operation of adding $\ell$ to each element of each subset modulo $n$. We use ${[n] \choose k}$ to denote the set of $k$-element subsets of $[n]$. Code for generating a maximal $\rho^\ell$-symmetric weakly separated collection for a given $k, n,$ and $\ell$ satisfying the conditions of Theorem~\ref{thm: intro_WSCs} can be found at \cite{PattonGithub}. 

Theorem~\ref{thm: intro_WSCs} is a direct generalization of \cite[Theorem 1.6]{PasqTZ2019}, which produces $\rho^k$-symmetric maximal weakly separated collections. Weakly separated collections that are $\rho^\ell$-symmetric also appear as crucial ingredients in \cite{Fraser2020}, where the author gives a recipe for generating maximal-with-respect-to-inclusion sets for any $k, n,$ and $\ell$ with $k\leq d$. As such, Theorem~\ref{thm: intro_WSCs} also gives a criterion for determining when Fraser's weakly separated collections are size $k(n-k)+1$. 

The main result of \cite{PasqTZ2019} is largely motivated by the correspondence between maximal weakly separated collections of ${[n]\choose k}$ fixed under addition by $k$ and self-injective Jacobian algebras. While we are unaware of a corresponding representation-theoretic generalization of their work to the $\ell\neq k$ setting, we do wish to highlight certain connections between the weakly separated collections we produce and the cyclic symmetry locus of the Grassmannian, as studied in \cite{Fraser2020}. By \cite[Theorem 1.13]{HughesRoy}, if the induced action of $\psi^\ell$ on the sheaf moduli $\FM(\La(k, n-k))$ is globally foldable, then Theorem~\ref{thm:intro_twist-spun_fillings} implies that $\FM(\Sigma_{\psi^\ell}(\La(k, n-k)))$ admits a (skew-symmetrizable) cluster structure with charts induced by fillings of $\Sigma_{\psi^\ell}(\La(k, n-k))$. In general, however, the action of $\psi^\ell$ or $\rho^\ell$ need not be globally foldable. In these instances, Fraser produces a conjectural \emph{generalized} cluster algebra structure---in the sense of \cite{ChekhovShapiroGeneralizedCluster}---on the top-dimensional connected component of the $\rho^\ell$ fixed locus of $\Gr(k, n)\cong \FM(\La(k, n-k))$. Instances of generalized cluster algebras from fillings of twist-spuns were first observed in \cite{HughesRoy} for a family of finite-type examples, but the construction provided here gives a systematic means of providing examples of possible generalized cluster phenomena. To our knowledge, these two works comprise the first hints of such phenomena appearing in low-dimensional contact and symplectic topology. The third author, together with Daping Weng, hopes to continue to study generalized cluster phenomenon in this context in future work.

The connections between cluster theory and contact geometry highlighted above also hint at a possible converse for Theorem~\ref{thm:intro_twist-spun_fillings}. Specifically, Theorem~\ref{thm: intro_WSCs} implies that no cluster seed of $\Gr(k, n)$ comprised entirely of Pl\"ucker coordinates is fixed by the action of $\rho^\ell$. This provides a possible obstruction for fillings of $\Sigma_{\psi^\ell}(\La(k, n-k))$ that are smoothly the mapping torus of the $\psi^\ell$ action on a filling $L$ of $\La(k, n-k)$. A generalization of \cite[Conjecture 1.16]{Hughes2021}---see also Remark 6.23 in loc.\ cit.---would then lead us to expect the following:

\begin{conjecture}\label{conj: intro_converse}
The twist spun $\Sigma_{\psi^\ell}(\La(k, n-k))$ is orientably exact Lagrangian fillable if and only if $k$ is congruent to $-1, 0,$ or $1$ modulo $d$.
\end{conjecture}

We conclude this introduction by discussing certain limitations of our construction. As in \cite{PasqTZ2019}, the algorithm we give to produce the weakly separated collections of Theorem~\ref{thm: intro_WSCs} does not produce all possible $\rho^\ell$-symmetric weakly separated collections of ${[n]\choose k}$, though different choices of input data may produce distinct outputs; see \cite[Remark 5.5]{PasqTZ2019} for more details. Similarly, as is already well known, not every cluster seed of $\Gr(k, n)$ arises from a plabic graph or weakly separated collection \cite{Fomin7}. This partly motivates our focus on Legendrian weaves, as they are able to realize non-Plucker cluster variables, hence realizing strictly more seeds than plabic graphs; see \cite[Remark 1.1]{CLSW2023} for further discussion. However, the question of whether Legendrian weaves realize all possible cluster seeds is still unknown outside of finite and affine Dynkin types \cite{Hughes2021, ABL22}. We do not attempt to answer it in this setting.

The remainder of this paper is organized as follows. In Section~\ref{sec: geometric_background} we give the necessary background on Legendrian knots and weaves, and we present the twist-spinning construction. In Section~\ref{sec: combinatorics_background}, we define weakly separated collections and plabic tilings, and describe the T-shift algorithm. Section~\ref{sec: Construction} describes the process of producing a filling of a twist-spun Legendrian given the input data of a $\rho^\ell$-symmetric maximal weakly separated collection. Finally, in Section~\ref{sec: algorithm}, we describe an algorithm for producing such $\rho^\ell$-symmetric maximal weakly separated collections of Theorem~\ref{thm: intro_WSCs} and verify that they exhibit the necessary properties. 

\subsection*{Acknowledgments}
This work began as an undergraduate research project through the Math+ program at Duke University, and we wish to thank the organizers, Heekyoung Hahn and Lenny Ng, for providing an enriching experience for the participants. We are also grateful to Jiajie Ma, who served as the graduate student mentor for our project and was instrumental in making sure everything ran smoothly. The Math+ program is supported by the Department of Mathematics at Duke, the Rhodes Information Initiative at Duke, and Duke's Office of the Dean of Academic Affairs. JH would also like to thank Daping Weng for his interest in the project and for sharing useful insights.

\section{Contact-geometric background}\label{sec: geometric_background}
We begin with the necessary background on Legendrian links and their exact Lagrangian fillings.
 The standard contact structure $\xi_{st}$ in $\R^3$ is the 2-plane field given as the kernel of the 1-form $\alpha_{\st}=dz-ydx$. A link $\La \subseteq (\R^3, \xi_{st})$ is Legendrian if $\La$ is everywhere tangent to $\xi_{st}$. As $\La$ can be assumed to avoid a point, we can equivalently consider Legendrians $\La$ contained in the contact 3-sphere $(\mathbb{S}^3, \xi_{st})$ \cite[Section 3.2]{Geiges08}. We consider Legendrian links up to Legendrian isotopy, i.e. ambient isotopy through a family of Legendrians. 
In this work, we will depict a Legendrian link $\La \subseteq (\R^3, \xi_{st})$ solely via the front projection $\Pi:(\R^3, \xi_{\st}) \to \R^2$ given by $\Pi(x, y, z)=(x, z)$.

The symplectization $\Symp(M, \ker(\alpha))$ of a contact manifold $(M, \ker(\alpha))$ is the symplectic manifold $(\R_t\times M, d(e^t\alpha))$. Given two Legendrian links $\La_-, \La_+\subseteq (\R^3,\xi_{\st})$,  an exact Lagrangian cobordism $L\subseteq \Symp(\R^3, \ker(\alpha_{\st}))$ from $\La_-$ to $\La_+$ is a cobordism $\Sigma$ such that there exists some $T>0$ satisfying the following: 

	\begin{enumerate}
		\item $d(e^t\alpha_{\st})|_\Sigma=0$ 
		\item $\Sigma\cap ((-\infty, T]\times \R^3)=(-\infty, T]\times \La_-$ 
		\item $\Sigma\cap ([T, \infty)\times \R^3)=[T, \infty) \times \La_+$ 
		\item $e^t\alpha_{\st}|_\Sigma=df$ for some function $f: \Sigma\to \R$ that is constant on $(-\infty, T]\times \La_-$ and $[T, \infty)\times \La_+$. 
	\end{enumerate}

	An {\bf exact Lagrangian filling} of the Legendrian link $\La\subseteq (\R^3, \xi_{\st})$ is an exact Lagrangian cobordism $L$ from $\emptyset$ to $\La$ that is embedded in $\Symp(\R^3, \ker(\alpha_{\st}))$. Equivalently, we consider $L$ to be embedded in the symplectic 4-ball with boundary $\partial L$ contained in contact $(\S^3, \xi_{\st})$.

\subsection{Legendrian twist-spun tori}
Let $\La\subseteq (\R^3, \xi_{\st})$ be a Legendrian link and $\varphi_t:\R^3\times [0, 1]\to \R^3$ be a {\bf Legendrian loop} of $\La$. That is, $\varphi_t$ is a Legendrian isotopy satisfying $\La=\varphi_0(\La)=\varphi_1(\La)$. For $t\in [0, 1]$, we obtain an $\S^1$ family of Legendrians $\{\varphi_t(\La)\}$ that is smoothly isotopic to the mapping torus of $\varphi_t$. This allows us to obtain the {\bf twist-spun Legendrian} $\Sigma_\varphi(\La)\subseteq (\R_z\times T^*\R_{x\geq 0}\times T^*S^1, dz-p_x d_x - p_\theta d\theta)$ as the unique Legendrian lift of our $S^1$-family. The canonical identification of $T^*\R_{x\geq 0}\times T^*S^1$ with $T^*\R^2$ via the map $R_{x\geq 0}\times S^1\to \R\times\R\backslash \{0\}$ given by $(x, \theta)\mapsto xe^{i\theta}$ gives a contact embedding $\R_z\times T^*\R_{x\geq 0}\times T^*S^1 \xhookrightarrow{} (\R^5, \xi_{\st})$; see \cite[Section 1.3]{DRG21} for more details. This allows for the following description of twist-spuns:

\begin{definition}
    Let $\varphi_t$ be a Legendrian loop of $\La\subseteq (\R^3, \xi_{\st})$. The twist-spun Legendrian $\Sigma_\varphi(\La)\subseteq (\R^5, \xi_{\st})$ is the union of Legendrian tori $$\La\times[0,1]/\La\times\{0\}\sim \varphi(\La)\times\{1\}.$$
\end{definition}

Given a twist-spun $\Sigma_\varphi(\La)$, one can construct an exact Lagrangian filling in the symplectization $\Symp(\R^5, \xi_{\st})$ by a similar mapping torus construction. In particular, if there exists an exact Lagrangian filling $L$ of $\La$, such that $L\cong L\cup_{\partial L} \{\varphi_t\}$, then the corresponding $\S^1$ family of exact Lagrangians forms a filling of $\Sigma_\varphi(\La)$ by \cite[Proposition 6.2]{HughesRoy}. We denote this filling by $L\times_\varphi S^1$. 

\subsection{Legendrian weaves}\label{sec:weaves}

We now describe Legendrian weaves, a construction of Casals and Zaslow that can be used to produce exact Lagrangian fillings of a Legendrian link \cite{CZ2022}. The key idea of their construction is to combinatorially encode a {\it Legendrian} surface $\w$ in the 1-jet space $J^1(\D^2)=T^*\D^2\times \R_z$ by the singularities of its front projection in $\D^2\times \R_z$. The Lagrangian projection $\pi:T^*\D^2\times \R_Z\to T^*\D^2$ of $\w$ then yields an exact Lagrangian surface in $T^*\D^2$.

More explicitly, we construct a filling of $\La$ by first describing a local model for a Legendrian surface $\w$ in $J^1(\D^2)=T^*\D^2\times \R_z$. We equip $T^*\D^2$ with the symplectic form $d(e^r\alpha)$ where $\ker(\alpha)=\ker(dy_1-y_2d\theta)$ is the standard contact structure on $J^1(\partial \D^2)$ and $r$ is the radial coordinate. This choice of symplectic form ensures that the flow of 
the Liouville vector field defined by the 1-form $e^r\alpha$ is transverse to $J^1(\S^1)\cong \R^2\times \partial \D^2$, thought of as the cotangent fibers along the boundary of the 0-section. The Lagrangian projection of $\w$ is then a Lagrangian surface in $(T^*\D^2, d(e^r\alpha))$. Moreover, since $\w\subseteq (J^1(\D^2), \ker (dz-e^r\alpha))$ is a Legendrian, we immediately obtain the function $z:\pi(\w)\to \R$ satisfying $dz=e^r\alpha|_{\pi(\w)}$, demonstrating that $\pi(\w)$ is exact. 

The boundary of $\pi(\w)$ is taken to be a positive braid $\beta$ in $J^1(\S^1)$  so that we may regard it as a Legendrian link in a contact neighborhood of $\partial \D^2$. As the 0-section of $J^1(\S^1)$ is Legendrian isotopic to a max-tb standard Legendrian unknot, we can take $\partial\pi(\w)$ to equivalently be the standard satellite of the standard Legendrian unknot. Diagramatically, this implies that the braid $\beta$ in $J^1(\S^1)$ can be given as the $(-1)$-framed closure of $\beta$ in $(\R^3,\xi_{\st})$.

\subsubsection{$N$-Graphs and Singularities of Fronts} To construct a Legendrian weave surface $\w$ in $J^1(\D^2)$,we combinatorially encode the singularities of its front projection in a colored graph. 
Local models for these singularities of fronts are classified by work of Arnold \cite[Section 3.2]{ArnoldSing}. The singularities that appear in our construction describe elementary Legendrian cobordisms and are pictured in Figure \ref{fig: wavefronts}.

	\begin{center}
		\begin{figure}[h!]{ \includegraphics[width=.8\textwidth]{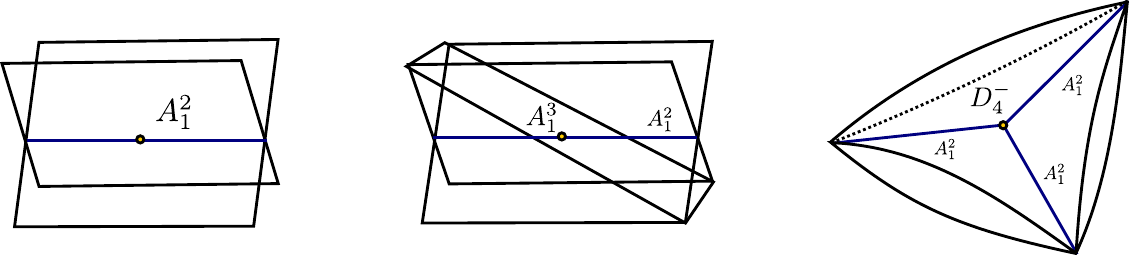}}\caption{Singularities of front projections of Legendrian surfaces. Labels correspond to notation used by Arnold in his classification.}
			\label{fig: wavefronts}\end{figure}
	\end{center}

Since the boundary of our singular surface $\Pi(\w)$ is the front projection of an $N$-stranded positive braid, $\Pi(\w)$ can be pictured as a collection of $N$ sheets away from its singularities. We describe the behavior at the singularities as follows:

\begin{enumerate}
    \item The $A_1^2$ singularity occurs when two sheets in the front projection intersect. This singularity can be thought of as the trace of a constant Legendrian isotopy in the neighborhood of a crossing in the front projection of the braid $\beta\Delta^2$. 
    \item The $A_1^3$ singularity occurs when a third sheet passes through an $A_1^2$ singularity. This singularity can be thought of as the trace of a Reidemeister III move in the front projection.
    \item A $D_4^-$ singularity occurs when three $A_1^2$ singularities meet at a single point. This singularity can be thought of as the trace of a 1-handle attachment in the front projection. 
\end{enumerate}

Having identified the singularities of fronts of a Legendrian weave surface, we encode them by a colored graph $\Gamma\subseteq \D^2$. The edges of the graph are labeled by Artin generators corresponding to the braid word $\beta$. In the interior, we require that any edges labeled $\sigma_i$ and $\sigma_{i+1}$ meet  at a hexavalent vertex with alternating labels while any edges labeled $\sigma_i$ meet at a trivalent vertex. 

To obtain a Legendrian weave $\w(\Gamma)\subseteq (J^1(\D^2),\xi_{\st})$ from an $N$-graph $\Gamma$, we glue together the local germs of singularities according to the edges of $\Gamma$. First, consider $N$ horizontal sheets $\D^2\times \{1\}\sqcup \D^2\times \{2\}\sqcup \dots \sqcup \D^2\times \{N\}\subseteq \D^2\times \R$ and an $N$-graph $\Gamma\subseteq \D^2\times \{0\}$. 
   
    \begin{center}		\begin{figure}[h!]{ \includegraphics[width=.8\textwidth]{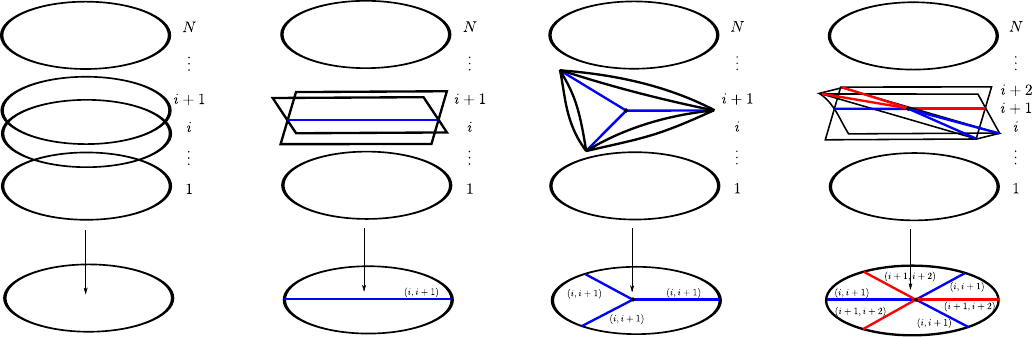}}\caption{The weaving of singularities of fronts along the edges of the $N$-graph. Gluing these local models according to the $N$-graph $\Gamma$ yields the weave $\w(\Gamma)$.}\label{fig:Weaving}\end{figure}
	\end{center}

	If we take an open cover $\{U_i\}_{i=1}^m$ of $\D^2\times \{0\}$ by open disks, refined so that any disk contains at most one of the features depicted in Figure~\ref{fig:Weaving}, we can glue together the corresponding fronts according to the intersection of edges along the boundary of the disks.

	\begin{definition}\label{def: weave}
	    The \textbf{Legendrian weave} $\w(\Gamma)\subseteq (J^1(\D^2), \xi_{st})$ is the Legendrian lift of the front $\Pi(\w(\cup_{i=1}^m U_i))$ given by gluing the local fronts of singularities together according to the $N$-graph $\Gamma$.
	\end{definition}

	The immersion points of a  Lagrangian projection of a weave surface $\w$ correspond precisely to the Reeb chords of $\w$. In particular, if $\w$ has no Reeb chords, then its Lagrangian projection $L(\w)$ is an embedded exact Lagrangian filling of the Legendrian link $\partial\w$. In the Legendrian weave construction, Reeb chords correspond to critical points of functions giving the difference of heights between sheets. Every weave surface in this work admits an embedding where the distance between the sheets in the front projection grows monotonically in the direction of the boundary, ensuring that there are no Reeb chords.

\section{Combinatorial background}\label{sec: combinatorics_background}
In this section, we cover the combinatorial background necessary to explain the construction we use to prove Theorem~\ref{thm:intro_twist-spun_fillings}. We start by describing weakly separated collections of the numbers $1$ through $n$ of size $k$. We then use them to define plabic tilings, and their dual graphs, more commonly known as plabic graphs. Finally, we describe a combinatorial recipe known as T-shift from \cite{CLSW2023} that allows us to obtain a Legendrian weave from these combinatorial data. 

\subsection{Weakly Separated Collections}

Recall that for $n \in \mathbb{N}$, we denote by $[n]$ the set $\{1, \dots, n\}$. We also denote by ${[n] \choose k}$ the set of all $k$-element subsets of $[n]$.

\begin{definition}\label{def:weakly-separated}
Two sets $S$  and $T$  in ${[n] \choose k}$ are \textbf{weakly separated} if there are no numbers $a < b < c < d$ all in $[n]$ such that $a, c \in S \backslash T$ and $b, d \in T \backslash S$. A collection of sets $\{U_i\}$ is called \textbf{weakly separated} if each $U_i$ is pairwise weakly separated.
\end{definition}

See \Cref{ex:WSC} below for an example.
We refer to a weakly separated collection $D\subseteq {[n]\choose k}$ as \textbf{maximal} if it is maximal by inclusion. Note that any maximal weakly separated collection necessarily contains all $n$ intervals of cyclically consecutive numbers $\{i, i+1, \dots, i+k-1\}$.

Given $k$ and $n$, the cardinality of a maximal weakly separated collection is fixed.

\begin{theorem}[Theorem 1.3, \cite{OPS2015}]\label{WSC-Size} A maximal weakly separated collection $D\subseteq {[n]\choose k}$ contains precisely $k(n-k)+1$ elements.
\end{theorem}

Our method for constructing Legendrian weaves uses weakly separating collections as a key tool. We introduce here a notion of symmetry for weakly separated collections that corresponds to rotational symmetry of Legendrian weaves. Given a set $S\subseteq [n]$, we denote by $S +_n \ell$ the subset of $[n]$ obtained from $S$ by adding $\ell$ to every element modulo $n$. 
\begin{definition}\label{def:symmetric-ws-collection}
A weakly separated collection $D$ is called $\rho^\ell$-\textbf{symmetric} if
 for any $I\in D,$ we have that $I+_n \ell$ is also in $D$. 
 
\end{definition}

\begin{ex}\label{ex:WSC}
    The set $D=\{123, 234, 345, 456, 156, 126, 136, 236, 346, 356\}$ is a $\rho^3$-symmetric maximal weakly separated subset of ${[6]\choose 3}$. Note that $123$ is shorthand for the set $\{1,2,3\}$. 
\end{ex}

\subsection{Plabic graphs and plabic tilings}
We now introduce plabic graphs and plabic tilings.
Plabic graphs are planar bicolored graphs originally introduced in \cite{postnikov2006} that have a rich combinatorial structure tied to cluster algebras and positroid varieties. We follow the same conventions as \cite{CLSW2023}. 

\begin{definition}\label{def: plabic_graph}
Let $\D_n$ be the unit disk with $n$ marked points on the boundary
A {\bf plabic graph} $G$ on $\D_n$ is a planar graph embedded in $\D_n$ such that: \begin{enumerate}
    \item each marked point of $\D_n$ is a boundary vertex of $G$,
    \item each internal (non-boundary) vertex of $G$ is colored black or white, and
    \item each boundary vertex is adjacent to a unique internal vertex of $G$.     
\end{enumerate}
\end{definition}

The faces of a plabic graph $G$ are the connected components of $\D_n\backslash G$. One can label each face by the elements of a weakly separated collection to represent cluster seeds of positroid varieties. The assignment of these face labels $\mathcal{F}(G)$ is often achieved via a collection of curves associated to $G$ called zig-zag strands. These combinatorial data are fixed by a collection of local moves.

\begin{definition}\label{def:plabic-equivalence}
We consider plabic graphs to be equivalent under the following changes: 
\begin{enumerate}
    \item deletion of a vertex from a pair of bivalent vertices, as in Figure~\ref{fig: plabic_equivalences} (left);
    \item contraction of adjacent vertices with the same color, as in Figure~\ref{fig: plabic_equivalences} (right). 
\end{enumerate}
\end{definition}

\begin{figure}[h!]{ \includegraphics[width=.8\textwidth]{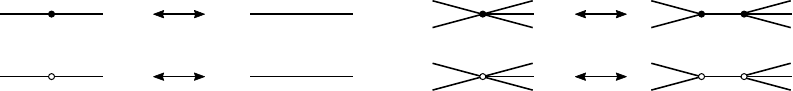}}\caption{Local pictures of equivalence moves on plabic graphs.}
			\label{fig: plabic_equivalences}\end{figure}

We also introduce the square move, which switches the coloring of vertices in a bipartite square; see Figure~\ref{fig: square_move}. The square move is sometimes included as an equivalence of plabic graphs, as it preserves a permutation associated to the zig-zag strands of $G$. It does not, however, preserve the face labels $\mathcal{F}(G)$ and we do not consider it to be an equivalence move in this work.  

        \begin{figure}[h!]{ \includegraphics[width=.3\textwidth]{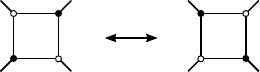}}\caption{Local picture of a square move on a plabic graph.}
			\label{fig: square_move}\end{figure}

We assume that all of our plabic graphs are reduced, meaning that they do not contain any monogons or bigons up to equivalence and applications of square moves; see \cite[Definition 7.1.6]{Fomin7} for a precise definition. We further assume that our plabic graphs have no lollipops, i.e., no internal vertices of degree 1. This is equivalent to assuming that our weakly separated collection $D\subseteq {[n]\choose k}$ is maximal and corresponds to a cluster seed in the Grassmannian $\Gr(k, n)$ rather than a seed in a positroid variety of lower dimension. 

 While it is relatively straightforward to obtain the face labels of a plabic graph from its zig-zag strands, our construction proceeds in the opposite direction: from a weakly separated collection, we wish to produce a plabic graph and a corresponding Legendrian weave. To that end, we describe an intermediate combinatorial object known as a plabic tiling, originally defined in \cite{OPS2015}, in order to pass from weakly separated collections to plabic graphs. We start by defining some auxiliary sets associated to a weakly separated collection. 

\begin{definition}\label{def:white-black-cliques}
Given a maximal weakly separated collection $D \subseteq \binom{[n]}{k}$, define for all $K \in \binom{[n]}{k-1}$, the associated \textbf{white clique} $\mathcal{W}(K) = \{S \in D \mid K \subseteq S\}$. Similarly, for all $L \in \binom{[n]}{k+1}$, define the \textbf{black clique} to be $\mathcal{B}(L) = \{S \in D \mid S \subseteq L\}$.
\end{definition}

White cliques are necessarily of the form $\mathcal{W}(K)=\{Ka_1, Ka_2, \dots, Ka_r\}$ for $a_1<a_2<\dots < a_r\in [n]$. Similarly, black cliques are necessarily of the form $\mathcal{B}(L)=\{L\backslash b_1, L\backslash b_2, \dots L\backslash b_s\}$ for $b_1<b_2< \dots < b_s\in [n]$. We call a (white or black) clique nontrivial if it contains at least three elements.

\begin{ex}
    The nontrivial white cliques of the weakly separated collection from \Cref{ex:WSC} correspond to the following 2-element subsets of $[6]$: $23, 56, 16, 34$, and $36$. The white clique $\mathcal{W}(36)$ is the set $\{136, 236,346, 356\}$. The nontrivial black cliques correspond to the following 4-element subsets of $[6]$: $1236, 3456, 2346, 1356.$ The black clique $\mathcal{B}(1356)$ is the set $\{136, 156, 356\}$. 
\end{ex}

The boundary of a nontrivial white clique $\mathcal{W}(K)$ is the following collection of pairs of elements in $\mathcal{W}(K):$ 
\begin{equation*}
    \partial \mathcal{W}(K)=\{(Ka_1, Ka_2), (Ka_2, Ka_3),\dots, (Ka_{r-1}, Ka_r), (Ka_r, Ka_1)\}
\end{equation*}
Similarly, the boundary of a nontrivial black clique $\mathcal{B}(L)$ is the following collection of pairs of elements in $\mathcal{B}(L):$

\begin{equation*}
    \partial \mathcal{B}(L)=\{(L\backslash b_1, L\backslash b_2), (L\backslash b_2, L\backslash b_3),\dots, (L\backslash b_{s-1}, L\backslash b_s), (L\backslash b_s, L\backslash b_1)\}
\end{equation*}

We can now define a 2-dimensional CW complex $\Sigma(D)$ associated to a maximal weakly separated collection $D$ following \cite{OPS2015}, with a slight modification introduced by \cite{PasqTZ2019} to take into account possible symmetries of $D$. Let $v_1, \dots, v_n \in \R^2$ be the vertices of a regular $n$-gon centered at the origin. For $S\subseteq [n]$, we denote $V_S=\sum_{i\in S} v_i$.  
\begin{definition}\label{def:plabic-tiling}
Given a maximal weakly separated collection $D$, a \textbf{plabic tiling} $\Sigma(D) \in \mathbb{R}^3$ has 
\begin{itemize}
    \item vertices given by $\{v_s \mid s \in D\}$,
    \item edges connecting $v_s$ and $v_t$ if $(s, t)$ appears in the boundary of some clique
    \item faces given by the condition that their boundary vertices form a nontrivial clique.
\end{itemize}
\end{definition}

See Figure~\ref{fig:plabic_tiling_ex} for an example. Note that, by \cite[Proposition 9.4]{OPS2015}, when $D$ is maximal, $\Sigma(D)$ forms a polygonal tiling of the convex $n$-gon with boundary vertices labeled by cyclic $k$-element interval subsets of $[n]$.

\begin{figure}
    \centering
    \includegraphics[width=0.45\linewidth]{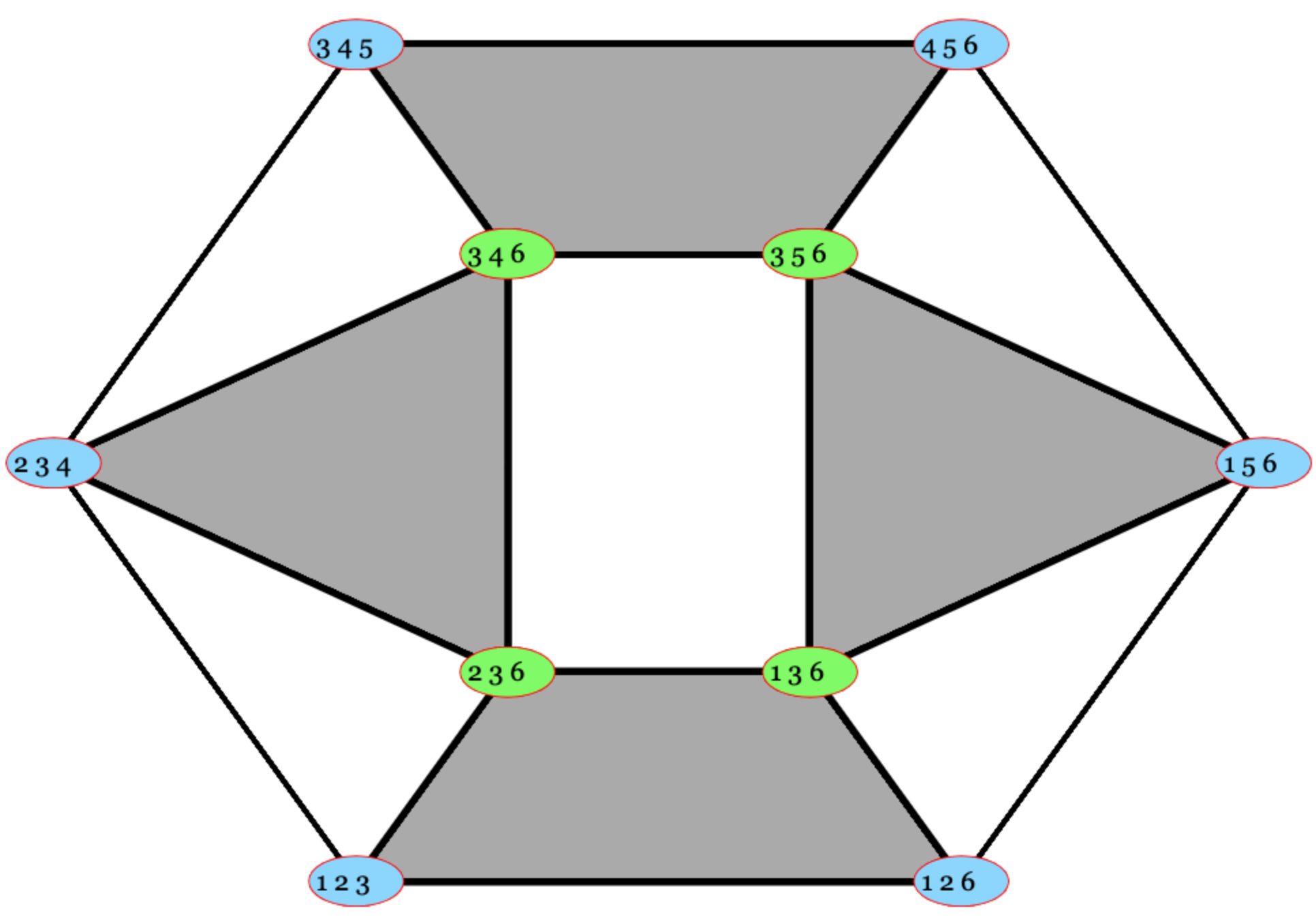}
    \caption{Plabic tiling $\Sigma(D)$ corresponding to the weakly separated collection $D$ from \Cref{ex:WSC}. Figure produced using \cite{PlabicTilingCalc}.}
    \label{fig:plabic_tiling_ex}
\end{figure}

From the construction of $\Sigma(D)$, we can see that the dual graph $G_D$ forms a plabic graph with white vertices dual to white cliques and black vertices dual to black cliques. The labels on the vertices of $\Sigma(D)$ correspond to the face labels of $G_D$.

\subsection{Grassmannians and the cyclic shift automorphism}\label{sub: Grassmannians}

Let $\Gr(k,n)$ denote the space of complex $k$-dimensional subspaces of $\C^n$ and $\widetilde{\Gr}(k, n)$ the corresponding affine cone. Given a point $X\in \Gr(k, n)$, thought of as the row span of a $k\times n$ matrix, the Pl\"ucker embedding of $\Gr(k, n)$ describes $X$ by its $k\times k$ minors. We denote by $\Delta_{I}$ with $I\in {[n]\choose k}$ the corresponding Pl\"ucker function. By \cite{Scott06}, the coordinate ring $\C[\widetilde{\Gr}(k, n)]$ is a cluster algebra with Pl\"ucker coordinates $\Delta_{I}$ appearing as a subset of the cluster variables. The reader unfamiliar with cluster algebras can consult \cite{FWZ1} for a detailed introduction, though we largely avoid any technicalities related to them in this work. 

If we view the elements of a weakly separated collection as indices of Pl\"ucker coordinates, then the combinatorics of plabic graphs give ways to understand the cluster structure on $\C[\widetilde{\Gr}(k, n)]$. In particular, we can relate the notion of rotational symmetry of a plabic graph or plabic tiling to a particular operation on $\widetilde{\Gr}(k, n)$ known as the cyclic shift automorphism. This automorphism, commonly denoted by $\rho$, acts on the Grassmannian $\widetilde{\Gr}(k, n)$ by shifting the indices of Pl\"ucker coordinates: $\Delta_{I}\mapsto \Delta_{I+_n 1}$. 
The action preserves the cluster structure, and is therefore a cluster automorphism \cite{fraser2018braid}. 

In our context, under the identification of weakly separated collections with Pl\"ucker coordinates, the cyclic shift acts by $I\mapsto I+_n 1$ for $I\in D$ a $k$-element set. As the boundary vertices of a plabic tiling (hence the boundary faces of a plabic graph) are labeled by $k$-element consecutive sets, this has the effect of rotating the plabic tiling (hence the dual plabic graph) by $2\pi/n$. Thus, a $\rho^\ell$-symmetric weakly separated collection fixed corresponds to a plabic graph with $2\pi \ell /n$ rotational symmetry. 

\subsection{T-shift of plabic graphs}
We now give a recipe for constructing a Legendrian weave from a trivalent plabic graph following \cite{CLSW2023}. Note that for plabic graphs without lollipops, one can apply the moves in Definition~\ref{def:plabic-equivalence} to show that any plabic graph is move equivalent to a trivalent plabic graph; see e.g. \cite[Lemma 7.4.2]{Fomin7}.

\begin{definition}\label{def:T-shift}
Given a trivalent plabic graph $G$, we define the $\bm{T}$\textbf{-shift} $G^\downarrow$ to be the trivalent plabic graph constructed via the following steps:
\begin{enumerate}
    \item For each marked point $i$ in $\partial \D$, add a new marked point $i'$ slightly counterclockwise of $i$.
    \item Replace every internal black trivalent vertex with a white trivalent vertex.
    \item Place a black vertex $u$ in each face of $G$. For each white vertex $v$ added in Step 2 that appears on the boundary of the face, draw an edge connecting it to $u$. For every $u$ in a boundary face, add an edge from the bordering marked point $i'$ to $u$.
    \item Delete every black vertex of degree 2. For every black vertex of degree $d\geq 3$, arbitrarily expand it into a trivalent tree with $d$ leaves. 
\end{enumerate}
\end{definition}

We refer to the rank of $G$ as the size of a face label $I\in \mathcal{F}(G)$. Note that this quantity is well-defined, as $\mathcal{F}(G)\subseteq {[n] \choose k}$.

\begin{proposition}[Proposition 3.3, \cite{CLSW2023}]
    If $G$ is a reduced trivalent plabic graph, then $G^\downarrow$ is also a reduced trivalent plabic graph with rank one less than that of $G$.
\end{proposition}

We now describe how to use repeated iteration of T-shift to produce a Legendrian weave. Fixing $G$ with $\mathcal{F}(G)\subseteq {[n]\choose k}$, we denote by $G_1=G^\downarrow, G_2, \dots, G_{k-1}$ the plabic graphs obtained from repeatedly applying T-shift to $G$. Label each $G_i$ by $\sigma_i$ and denote their union by $\w(G)=\cup_{i=1}^{k-1} G_i$. 
\begin{proposition}[Proposition 3.7, \cite{CLSW2023}]\label{prop: T-shift_weave}
    Given a plabic graph $G$ with face labels $\mathcal{F}(G)\subseteq {[n]\choose k}$, the graph $\w(G)$ is a $k$-graph encoding a Legendrian weave embedded in $J^1(\D^2)$ with boundary braid $(\sigma_1\dots \sigma_{k-1})^n$.
\end{proposition}

See Figure~\ref{fig:TshiftEx} for the T-shift construction applied to our running example of a $\rho^3$-symmetric weakly separated collection $D\subseteq {[6]\choose 3}$. 

\begin{figure}[h]
    \centering
\includegraphics[width=.7\linewidth]{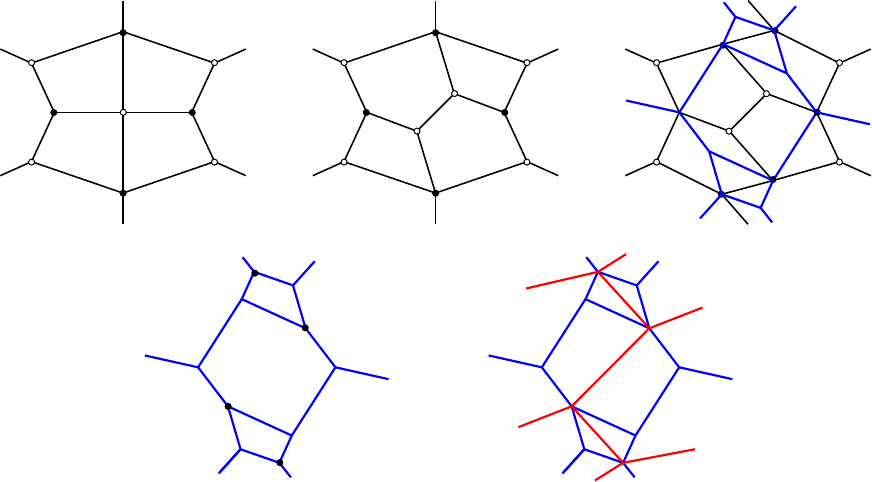}\caption{ The $T$-shift construction applied to the plabic graph dual the plabic tiling from Example \ref{fig:plabic_tiling_ex}. The top row depicts the process of obtaining $G^\downarrow_D$, while the bottom row depicts the second iteration of this process, which results in a Legendrian weave filling of $\La(3,3)$.}
\label{fig:TshiftEx}
\end{figure}

\begin{remark} Note that, as observed in \cite{CLSW2023}, the T-shift procedure does not produce all possible Legendrian weave fillings of $\La(k, n-k)$; see Remark 1.1 of loc.\ cit. for further details. 
\end{remark}

\section{Constructing fillings of twist-spun torus links}\label{sec: Construction}

In this section, we prove Theorem~\ref{thm:intro_twist-spun_fillings} assuming the existence of a $\rho^\ell$-symmetric maximal weakly separated collections, which we then construct in Section~\ref{sec: algorithm}. We start by verifying that $\rho^\ell$-symmetric weakly separated collections yield rotationally symmetric plabic tilings, and hence, rotationally symmetric Legendrian weaves.

\begin{lemma}\label{lemma:symmetric_plabic_tiling}
    Let $D\subseteq {[n] \choose k}$ be a $\rho^\ell$-symmetric maximal weakly separated collection. The plabic tiling $\Sigma(D)$ has $2\pi\ell/{n}$ rotational symmetry and is fixed by the action of $\rho^\ell$. 
\end{lemma}

\begin{proof}
For a $\rho^\ell$-symmetric maximal weakly separated collection $D$, we have that for any $I\in D,$ the $k$-element set $I+_n \ell$ is also in $D$. By Definition~\ref{def:plabic-tiling}, the vertices of $\Sigma(D)$ are given by $v_S=\sum_{i\in S} v_i$ where $\{v_i\}_{i=1}^n$ are the vertices of a regular $n$-gon centered at the origin. Therefore, for any vertex $v_I$ of $\Sigma(D)$, we must have that the point $v_{I+_n\ell}=\rho^\ell\cdot v_I$ is also a vertex of $\Sigma(D)$. Similarly, $\rho^\ell$-symmetry of $D$ implies that the $\rho^\ell$ orbits of the edges and faces of $\Sigma(D)$ also lie in $\Sigma(D)$.  

\end{proof}

For a maximal weakly separated collection $D$, we denote by $G_D$ the reduced plabic graph dual to the tiling $\Sigma(D)$. By \cite[Lemma 5.8]{Hughes2024} the induced action of the K\'alm\'an loop $\psi$ on $\FM(\La(k, n-k)$ is identical to the action of the cyclic shift on $\Gr(k, n)$ and corresponds to a $2\pi/n$ rotation of any weave filling $\w$ of $\La(k, n-k)$. We now show that $\rho^\ell$-symmetry of $D$ extends to rotational symmetry of the corresponding Legendrian weave.

\begin{lemma}\label{lemma: symmetric_weave}
For any $\rho^\ell$-symmetric maximal weakly separated collection $D$, the Legendrian weave $\w(G_D)$ is Legendrian isotopic relative to the boundary to $\psi^\ell\cdot \w(G_D)$. 
\end{lemma}

\begin{proof}
Let $D \subseteq \binom{[n]}{k}$ be a $\rho^\ell$-symmetric maximal weakly separated collection. By Lemma~\ref{lemma:symmetric_plabic_tiling}, $\Sigma(D)$ is fixed by the $2\pi\ell/n$ rotation induced by the action of $\rho^\ell$, hence its dual plabic tiling $G_D$ must be as well. Following Proposition~\ref{prop: T-shift_weave}, we obtain an embedded Legendrian weave $\w(G_D)$ from the plabic graph $G_D$. To verify symmetry of $\w(G_D)$, we examine each black vertex $v\in G_D$ of degree $d > 3$. For each such vertex, we need to choose a trivalent tree $\tau_v$ with $d$ leaves to resolve $v$ into a tree of trivalent vertices. For any such $v$ not fixed by $\rho^\ell$, we simply resolve each vertex in the $\rho^\ell$ orbit of $v$ in exactly the same way. If $v$ is fixed by $\rho^\ell$, then we arbitrarily resolve it into $\tau_v$. Let $\tau$ denote a choice of such resolutions, and let $G_D(\tau)$ be the corresponding T-shifted plabic graph under such resolution. By \cite[Theorem 3.10 (3)]{CLSW2023}, different choices of $\tau$ result in plabic graphs $G_D(\tau)$ that are related by local equivalences. Therefore, $G_D(\tau)$ is related to $\rho^\ell\cdot G_D(\tau)$ by local equivalences, and hence the corresponding Legendrian weaves $\w(G_D(\tau))$ and $\psi^\ell\cdot \w(G_D(\tau))$ are Legendrian isotopic relative to their boundaries.
\end{proof}

\begin{remark}
 While our proof of Theorem~\ref{thm:intro_twist-spun_fillings} below only requires Legendrian weaves fixed by $\psi^\ell$ up to Legendrian isotopy, it appears to be possible to produce a version of a Legendrian weave that is fixed pointwise by using the notion of degenerate $N$-graphs from \cite{ABL22}. In loc.\ cit., the authors slightly expand the class of singularities one is allowed to consider in order to obtain symmetric weaves corresponding to non-simply-laced finite and affine Dynkin type cluster algebras. A similar approach in this setting would involve carefully analyzing the class of singularities that appear in the local isotopies relating the T-shifts of different resolutions of $d$-valent black vertices of $G_D$.
\end{remark}

Using the Legendrian weave described in Lemma~\ref{lemma: symmetric_weave}, we now construct a filling of the twist-spun $\Sigma_{\psi^\ell}(\La(k, n-k))$ in order to prove Theorem~\ref{thm:intro_twist-spun_fillings} assuming the existence of a maximal $\rho^\ell$-symmetric weakly separated collection. 

\begin{proof}[Proof of Theorem~\ref{thm:intro_twist-spun_fillings}]
Let $d=n/\gcd(n, \ell)$ and assume that $k$ is in the set of $\Z/d\Z$ congruence classes represented by $\{0, -1, 1\}$. By Theorem~\ref{thm: intro_WSCs}, we can construct a maximal weakly separated collection $D\subseteq {[n] \choose k}$ fixed by addition of $\ell$ modulo $n$. By Lemma~\ref{lemma: symmetric_weave}, the weave $\w(G)$ obtained from applying T-shift to $G$ is fixed up to Legendrian isotopy by the action of $\psi^\ell$. In order to obtain a filling of $\Sigma_{\psi^\ell}(\La(k, n-k))$, we compose the family of fillings $L\times_{\psi^\ell}[0, 1]$ with the trace of the isotopy between $\w(G)$ and $\psi^\ell\cdot \w(G)$, producing an $\S^1$ family of fillings of $\La(k, n-k)$. Applying \cite[Proposition 6.2]{HughesRoy} to this $\S^1$ family yields the desired filling of $\Sigma_{\psi^\ell}(\La(k, n-k))$.
\end{proof}

\section{Symmetric weakly separated collections.}\label{sec: algorithm}

In this section, we establish our necessary and sufficient condition for the existence of a $\rho^\ell$-symmetric maximal weakly separated collection. We start by showing that our condition is necessary by studying the fixed points of plabic tilings under rotation. We then give an algorithm generalizing \cite[Theorem 1.6]{PasqTZ2019} that constructs a $\rho^\ell$-symmetric maximal weakly separated collection given input data satisfying our condition.

\subsection{Necessary condition}
Throughout this section, we assume that $k\leq \frac{n}{2}$. We can do so without loss of generality because, for any maximal weakly separated collection $D\subseteq{[n]\choose k}$, the weakly separated collection $D'\subseteq{[n] \choose n-k}$ given by $D'=\{[n]\backslash A|  A \in D\}$ produces an isomorphic plabic tiling. The following lemma appears in \cite{PasqTZ2019} as a crucial technical ingredient in the proof of their necessary condition for the existence of a $\rho^k$-symmetric maximal weakly separated collection. We include a proof here for completeness.

\begin{lemma}[Lemma 2.2, \cite{PasqTZ2019}] \label{lemma: PTZ lemma 2.2}
    For $I\in {[n]\choose k}$, if $I=I+_n \ell$, then $d|k$. 
\end{lemma}
\begin{proof}
    Consider the map $\phi: \Z / n\Z \rightarrow \Z / n\Z $ given by $\phi(a)= a + \ell \pmod n$. The order of $\phi$ is $d$, as this is the smallest integer such that $d \ell \equiv 0 \pmod{n}$, and every element of $\Z / n\Z$ lies in a $\phi$-orbit of size $d$. Since $I$ is a subset of $\Z / n\Z$, and $I = I +_n \ell$, it follows that $I$ is a union of $\phi$-orbits. Since each $\phi$-orbit has size $d$, the cardinality $|I| = k$ must be a multiple of $d$. Thus, $d|k$.
\end{proof}

Our necessary condition now follows as a generalization of \cite[Proposition 2.1]{PasqTZ2019}.

\begin{proposition}\label{prop: necessary}
If $D\subseteq {[n]\choose k}$ is a $\rho^\ell$-symmetric maximal weakly separated collection, then $k \equiv c \pmod {d}$, for $c\in\{-1, 0, 1\}$.
\end{proposition} 
\begin{proof}

Let $D\subset \binom{[n]}{k}$ be a $\rho^\ell$-symmetric maximal weakly separated collection, and let  $\Sigma(D)$ be the associated plabic tiling. From the $\rho^\ell$-symmetry of $D$ we know that the plabic tiling $\Sigma(D)$ is fixed by $2\pi\ell/n$ rotation, and therefore $\rho^\ell\cdot \Sigma(D)= \Sigma(D)$. Consider the center of $\Sigma(D)$, which is the unique fixed point of $\rho^\ell$. There are three cases that can happen:

\begin{enumerate}
    \item If the center is a vertex $v_I$, then $I=I+_n \ell$. By \Cref{lemma: PTZ lemma 2.2}, we have $d \mid k $, and thus $k \equiv 0 \pmod{d}$.
    \item If the center lies on an edge of $\Sigma(D)$, then we must have $d=2$. However, if this edge is fixed by a rotation through an angle of $\pi$, then $\rho^\ell$ must send white cliques to white cliques, and black cliques to black cliques. In order for this to happen, then the two faces on either side of the edge fixed by $\rho^\ell$ must be the same color. But we cannot have this because white and black cliques must alternate. Thus this case cannot occur. 
    \item If the center lies in the interior of a face on the plabic tiling $\Sigma(D)$, then the face $F$ and its label $K$ must be fixed by the $\rho^\ell$ action. Since $F$ corresponds to either a black or a white clique, $K$ is either size $k-1$ or $k+1$. 
    By \Cref{lemma: PTZ lemma 2.2}, $d \mid |K|=k\pm 1$. Since $|K|$ is either $k+1$ or $k-1$, we get 
    that $k \equiv \pm1 \pmod{d}$.
\end{enumerate}

Taking all three cases, we have $k \equiv c\pmod{d}$ for $c \in \{-1,0,1\}$, which completes the proof.
\end{proof}

\subsection{Generating a set when $n=d\ell$}\label{sec: n=dl} We now turn to generating $\rho^\ell$-symmetric maximal weakly separated collections. In \cite[Section 5]{PasqTZ2019}, the authors detail an algorithm for generating $\rho^\ell$-symmetric maximal weakly separated collections when $k= \ell$. Here we give a generalization that covers the $k \neq \ell$ cases. We start by describing the process for obtaining a candidate weakly separated collection when $n=d\ell$. See Subsection~\ref{sub: not_divide} for a modification of the algorithm described below in the case where $n\nmid \ell$. 
\begin{enumerate}
    \item Assume $n=d\ell$ and write $k=dr+c$ for some $0\leq c\leq d-1$. For each $a \in [\ell]$, denote the $\rho^\ell$ orbit of $a$ in $[n]$ by 
    \begin{align*}
        \overline{a} = \{ a+i\ell \mid 0\leq i\leq d-1\}.
    \end{align*}
    \item Fix an arbitrary total order on these orbits: say $\overline{a_1} < \overline{a_2} < \dots <\overline{a_{\ell}}$. Denote the complement of the union of the first $s-1$ orbits of $a_i$ in $[n]$ by
    \begin{align*}
        P_s = [n] \backslash  \bigcup_{i=1}^{s-1} \overline{a_i}.
    \end{align*}
    \item For $P_s =\{ x_1< x_2 <...<x_m\}$, define the successor function $S_{P_s}(x_i) = x_{i+1}$ for $1 \leq i < m$, and $S_{P_s}(x_m) = x_1$. In other words, $S_{P_s} $ is a bijection from $P_s$ to itself that defines a cyclic ordering on $P_s$. 

 \item  For an additional parameter $h\in [d]$ we modify $P_s$ in the following manner: \begin{align*}
P_{s,h}&=P_s\backslash\{a_s+h\ell,\dots,a_s+(d-1)\ell\}
    \end{align*}
    for $1\leq h\leq d-1$. For $h=d$, we set $P_{s, d}= P_s$.  Denote the corresponding interval of consecutive elements of $P_{s, h}$ by
    $$I(i,h)=\{i,S_{P_{s,h}}(i),S^2_{P_{s,h}}(i),\dots,S^{k-1}_{P_{s,h}}(i)\}$$ for $i\in [S_{P_s}(a_s-\ell), a_s]$.

    \item Define $\rho^\ell$-orbit representatives by
   $$
        B_s = \{I(i,h) \mid i \in [S_{P_s}(a_s-\ell), a_s], 1\leq h\leq d,|I(i,h)|=k\} 
    $$
    and denote the corresponding $\rho^\ell$ orbits by 
     \begin{align*}
        L_s = \{I +_n x \ell \mid I \in B_s, x \in [d]\}
    \end{align*}

    \item Finally, we obtain a candidate weakly separated collection from the union of all $\rho^\ell$ orbits: 
    \begin{align*}
        D= \bigcup_{s=1}^{\ell-r+1} L_s
    \end{align*}

\end{enumerate}

We now give an alternative informal process for generating the sets $B_s$ that the reader may find slightly more intuitive. Note that this process allows for elements in the same $\rho^\ell$ orbit appearing in a given $B_s$, making enumeration more challenging. The formal process avoids this redundancy, and we will exclusively rely on it for proving Theorem~\ref{thm: intro_WSCs}. We assert without proof that the two processes yield the same weakly separated collection $D$ for a fixed $k,n,\ell$ and total ordering.

After performing Steps (1) -- (3) exactly as given above, we proceed as follows:
\begin{enumerate}
    
\item For a given $s$, consider an interval subset $I_1$ of $P_s$ of size $k$ ending in $a_s$. If no such interval exists, then $B_s$ is empty. Otherwise, append $I_1$ to $B_s$.

\item  Check if there are elements in $I_j$ to the right of $a_s$ that are also in $\overline{a_s}$. Remove the rightmost such element, if it exists, and append $S_{P_s}(P_s \backslash (I \cup \overline{a_s})$ to the right. Continue until no such elements remain, adding the resulting $I_{j+1}$ to $B_s$ each time an element is replaced. 

  \item If $I_j$ does not contain any elements in the congruence class $\overline{a_s}$ to the right of $a_s$, remove the leftmost element of $I_j$. Each time an element is removed, if the removed element is in $\overline{a_s}$, append the first missing successor to $a_s$ in $S_{P_s}$ that is not in $I_j\cup \overline{a_s}$ to the right. Otherwise, append the first successor to $a_s$ in $S_{P_s}$ not in $I_j$. Add the resulting set $I_{j+1}$ to $B_s$. 
    
\item Repeat Steps (2) and (3) until $a_s$ appears on the left of $I_j$ in Step (3). 

\item For every $i \in [S_{P_s}^{-(\ell-1)}(a_s), a_s]$, repeat Steps (1)-(3), replacing $I_1$ with $$I_j=\{i, S_{P_s}(i), S_{P_s}^2(i), \dots S_{P_s}^{k-1}(i)\}.$$ 
\end{enumerate}

Having obtained the collections $B_s$, we define the set $D$ as the union of $\rho^\ell$ orbits of $B_s$, as above.

\begin{ex}
\label{algo}
Here we consider the example $n=6, k= 3 ,\ell = 3$ and run the algorithm using the informal process. The reader can compare this to the formal process used in the following example. First, choose the ordering $\overline{3}<\overline{2} < \overline{1}$. So, in $B_1$ we start with $\{123\}$. We remove from the left as there are no other elements in $\overline{3}$ besides $3$ itself. We then get $\{234\}$ and $\{345\}$ as our our remaining intervals. There is nothing left to remove, so we proceed. For $B_2$, we select the intervals from $[6]\backslash\{3,6\}$ containing $a_2=2.$ As such, we start with $\{512\}$. Removing from the left we remove $5$. Note that since $5 \in \overline{2}$, we add the set $\{124\}$ to $B_2$, as $4\not\in \overline{2}$. Continuing removing from the left, we add $\{245\}$ to $B_2$ as well. We now have an element to the right of $2$ in $\overline{2}$ to remove, namely $5$. As such, we remove $5$ and add in the next non-equivalence class element $1$, giving us $\{241\}.$ Notice this process allowed for repeat seeds and elements in orbits of other seeds as $\{124\}=\{241\}$ and $\{125\}+_6 3=\{245\}$. With no elements left to remove, we continue to $B_3$. We can only select elements from $[6]\backslash\{2356\}$, which does not have $k=3$ elements. Thus, $B_3$ is empty. Therefore, a $\rho^3$-symmetric maximal weakly separated collection generated by the algorithm is $$ \{123, 234,345, 456, 156, 126,125,245,124,145 \}$$
\end{ex}
\begin{ex}
\label{ex:formal_algo}
Here we also consider the example $n=6, k= 3 ,\ell = 3$ and run the algorithm using the formal process. Choose the ordering $\overline{3}<\overline{2} < \overline{1}$. We have $P_1=[n]$. $B_1$ consists of intervals $I(i, h)$ in $P_1$ with $i\in[S_{P_1}(3-3),3]=[1,3]$. Therefore $B_1=\{123,234,345\}$. Since no interval contains $a_1+ h\ell=3+3=6,$ we see that $I_{i, h}$ will just produce copies of $I_i$. Next we generate $B_2$. Note $P_{2,1}=\{1,2,4\}$ and $P_{2,2}=\{1,2,4,5\}=P_2$. We consider intervals beginning at all $i\in[S_{P_2}(2-3), 2]= [1,2]$ for both values of $h$. This yields $B_2=\{124, 245\}$. Finally we see that $P_3=\{1,4\},$ implying that $B_3$ is empty. Therefore, a $\rho^3$-symmetric maximal weakly separated collection generated by the algorithm is $$ \{123, 234, 345, 456, 156, 126, 125, 245, 124, 145 \}$$ which is identical to the previous example. 
\end{ex}
\subsection{$D$ is maximal} We now show that the collection $D\subseteq {[n] \choose k}$ defined above has $k(n-k)+1$ elements. Our proofs in this subsection largely follow the arguments in \cite[Section 4]{PasqTZ2019}, inserting $\ell$ for $k$ where appropriate.
\begin{lemma}\label{lemma: sizeLsSmalls}
Let $r$ satisfy $k=dr+c$ for $c \in \{-1, 0, 1\}$. Then, if $s \leq \ell - r$, $|L_s|=d|B_s|$. As a consequence, $L_s$ does not contain any elements fixed by $\rho^\ell$ for $s\leq \ell -r$.
\end{lemma} 

\begin{proof}
 Let $I \in B_s$, and define $J=I \cap \overline{a_s}$. Note that the sets $J+_n \ell x$  for $1 \leq x \leq d$ are not all distinct only when $J=\overline{a_s}$. By construction, it follows that 
if $I \neq P_s$, then the sets $I+_n \ell x$ form $d$ distinct intervals within $P_s$, implying that in this case$|L_s|=d|B_s|$.

Since $P_s$ is fixed under addition of $\ell$ modulo $n$, we have that if $I=P_s$, then $I$ is also fixed by addition modulo $n$. 
For $|L_s| \neq d |B_s|$,  we must therefore have $|I|=|P_s|=k$ for some $I$. Assume towards a contradiction that this occurs. By construction, $|P_s|=n-d(s-1)=d(\ell-s+1)$. Since $k$ is then a multiple of $d$, we have $c=0$ and $r=\ell-s+1$. When $s \leq \ell-r$, the equality $r=\ell - s + 1$ implies that $s \leq s-1$, a contradiction. Thus, $|L_s|=d|B_s|$. 
\end{proof}

\begin{lemma}\label{lemma: SizeBsSmalls}
 For all $s < \ell-r$, $|B_s|=k$. Furthermore, if $k=dr-1$ or $k=dr$, $|B_{\ell-r}|=k$.
 \end{lemma}
\begin{proof}
Let $t=\ell-s+1$. We wish to enumerate distinct intervals $I(i,h) \subseteq P_{s,h}$ appearing in $B_s$. Since $|P_s|=dt,$ there is a bijection $f:P_s\to [dt]$ uniquely defined by $f(a_s)=t$ and $f(S_{P_s}(b))=S_{[dt]}(f(b))$ where the cyclic ordering on $[dt]$ is given by $S_{[dt]}(x)=x+_n 1$. We therefore assume without loss of generality that $P_s=[dt]$. To avoid double counting any particular $I(i, h)$, we introduce the notation $h^*$ to represent, for a fixed $i$, the smallest value of $h$ such that $I(i,h) = I(i, h^*)$. By construction, we see that $th^* \in I(i,h^*)$. Then $th^*-i+1$ is the size of the interval $[i, th^*]$, so it is the minimum size of $I(i, h)$ for a fixed $i$. Since $I(i,h)$ must have size $k$, we obtain the inequality $k \geq th^-i+1$. Rewriting the inequality, we define  

\begin{align*}
    \gamma_i= \left \lfloor \frac{k+i-1}{t} \right \rfloor  \geq h
\end{align*}

as the number of valid and distinct intervals $I(i, h^*)$ for our fixed $i$. This implies that when $|P_{s,h}|> k$, every $i \in [t]$ corresponds to $\gamma_i$ distinct values of $h$ that yield distinct intervals of size $k$. Therefore, for all $s$  such that $|P_{s,h}|>k$, we have that summing over $i \in [t]$ yields the total number of distinct intervals, which is also equivalent to the cardinality of $B_s$:

\begin{align*}
    |B_s| &= \sum_{i=1}^t \left \lfloor \frac{k+i-1}{t} \right \rfloor.  
\end{align*}

Let $k=at+b$, for integers $a \geq 0$, and $0 \leq b \leq t$. We then have

\begin{align*}
     \sum_{i=1}^t \left \lfloor\frac{k+i-1}{t} \right \rfloor 
   & = \left \lfloor  \sum_{i=1}^t \frac{at + b + i-1}{t} \right \rfloor \\
   &= at +  \left \lfloor  \sum_{i=1}^t \frac{b+i-1}{t} \right \rfloor \\
   &= at +\sum_{i=1}^t
    \begin{cases}
        0 & i < t-b\\
        1 & i \geq t-b
    \end{cases} \\
    &= at +b.
\end{align*}

Therefore, in order to describe when $|B_s|=k$, it suffices to determine when $|P_{s,h}|>k$. Recall that $|P_{s,h}| = d\ell-ds+h$, so that for any $s<\ell-r$, we have $|P_{s,h}|>k$. Additionally, if we assume that $|P_{s,h}|=k=dr-1$, then $h=d-1$ and $s=\ell-r+1$. Similarly, for $|P_{s,h}|=k=dr$, we have $h=d$, and $s = \ell-r+1$.
Thus, for all $s < \ell-r$, $|B_s| = k$. Moreover, for both $k=dr-1 $ and $k=dr$, we have $|B_{\ell-r}| = k$ as well. 
\end{proof}

\begin{corollary}\label{cor: sizeBl-r}
If $k=dr+1$, then $|B_{\ell-r}|=k-r$. 
\end{corollary} 
\begin{proof}
Let $k= dr+1$. Note that by construction, $P_{s,h}$ has $d\ell-ds+h$ elements. Therefore, we have $|P_{\ell-r,h}| > k$ when $h \geq 2$ and the number of these elements is $\sum \limits_{i=1}^{\ell-s+1} (\gamma_i -1)$. The case for $h=1$ then yields one additional set. Thus, when $s=\ell-r$, we have $t=r+1$ and the number of elements in $B_{\ell-r}$ is given by
\begin{align*}
    |B_{\ell -r}| &= 1+\sum_{i=1}^{r+1} (\gamma_i - 1) \\
    &= 1 + k -(r+1)\\
    &= k-r.
\end{align*}
\end{proof}

\begin{ex}
Again consider our running example of $n=6$, $k=3$, $\ell=3$. Note that in this case, $d=2$, $r=1$, and $k=dr+1$. Choose $s=1 < 3-1=2$, then $t=3-1+1=3$. Then $|B_1|=\sum_{i=1}^t \gamma_i = \sum_{i=1}^3 \left\lfloor \frac{k+i-1}{t} \right\rfloor = 1+1+1 = 3 =k$. 
For $s=2$, $|B_2|=\sum_{i=1}^2 \gamma_i=1+1=2=k-r$, matching the count from Example~\ref{algo}. 
\end{ex}

\begin{ex}
    Now let $n=6$, $k=2$, and $\ell=3$. In this case we have $d=2$, $r=1$, and therefore $k=dr$. We have $t=1$, so $|B_3|= \gamma_1 = \lfloor\frac{2+1-1}{1}\rfloor = 2 = k$. 
\end{ex}

\begin{lemma}
\label{Lemma:ModImps}

 The number of elements in $L_{\ell-r+1}$ is given by 
 \begin{equation*}
     |L_{\ell-r+1}|=\begin{cases}
         k+1 & c=-1\\
         1 & c=0\\
         0 & c=1.
     \end{cases}
 \end{equation*}
 \end{lemma}
\begin{proof}
Recall that $B_{\ell-r+1}$ corresponds to the final value of $s$ for which $P_s$ could possibly contain $k$ elements. We consider the three cases of $k=dr+c$ for $c\in \{-1, 0, 1\}$ separately:
\begin{enumerate}
    \item If $k=dr-1$, then the set $P_{\ell-r+1}$ has size $dr=k+1$. Therefore, since removing any element of $P_{\ell-r+1}$ yields a size $k$ subset of consecutive elements, the elements $I(i, h)$ of $B_{\ell-r+1}$ are therefore all $k+1$ of the $k$-element subsets of $P_{\ell-r+1}$. Moreover, since $P_{\ell-r+1}=P_{\ell-r+1}+_n \ell$ by construction, it follows that the set of $k$-element subsets of $P_{\ell-r+1}$ is fixed by addition modulo $n$ as well. Thus, $|B_{\ell-r+1}|=|L_{\ell-r+1}|=k+1$.
    \item If $k=dr$, then $|P _s|=dr=k$, so there is exactly one $k$-element subset, and it is fixed by addition of $\ell$ modulo $n$. Thus, $|B_{\ell-r+1}|=|L_{\ell-r+1}|=1$.
    \item If $k=dr+1$, then $|P_s|=k-1$ and cannot contain any $k$-element subsets. Thus $B_{\ell-r+1}=L_{\ell-r+1}=\emptyset$.
\end{enumerate}

\end{proof}

\begin{proposition}\label{prop: cardinality_n=dl}  The number of elements in $D$ is $k(d\ell-k)+1$. 
\end{proposition} 
\begin{proof}
We apply a combination of Lemma~\ref{lemma: sizeLsSmalls}, Lemma~\ref{lemma: SizeBsSmalls}, Corollary~\ref{cor: sizeBl-r}, and Lemma~\ref{Lemma:ModImps} to determine the size of $D=\bigcup_{s=1}^{\ell-r+1}L_s$ in the following cases:
\begin{enumerate}
    \item If $k \equiv -1 \pmod d$, we have: 
\begin{align*}
    |D|&= kd(\ell-r)+k+1 \\
    &= kd\ell-kdr+k+1 \\
    &= kd\ell-k(k-1)+k+1 \\
    &= kd\ell-k^2+1 \\
    &= k(d\ell-k)+1
\end{align*}
\item If $k \equiv 0 \pmod d$: 
\begin{align*}
    |D|&= kd(\ell-r)+1 \\
    &= kd\ell-kdr+1 \\
    &= kd\ell-k^2+1 \\
    &= k(d\ell-k)+1
\end{align*}
\item Lastly, if $k \equiv 1 \pmod d$: 
\begin{align*}
 |D|   &= kd(\ell-r-1)+d(k-r) \\
    &= kd \ell - kdr - kd+kd-dr \\
    &= kd\ell - (k+1)(k-1) \\
    &= kd\ell-k^2+1 \\
    &= k(d\ell-k)+1
\end{align*}
\end{enumerate}

\end{proof}

\subsection{$D$ is weakly separated}

In this subsection, we prove that the collection $D\subseteq {[n] \choose k}$ generated by the algorithm given in Subsection~\ref{sec: algorithm} is weakly separated. As in the previous subsection, we follow a similar strategy to \cite{PasqTZ2019}, inserting $\ell$ for $k$ where appropriate. Where the statements are independent of $\ell$, we provide a sketch of their proofs for clarity.

\begin{lemma}[Lemma 3.3, \cite{PasqTZ2019}]\label{lemma: bordinI}
 Suppose we have a set $P$ with  cyclic ordering $<_P$ and an interval $I$ of consecutive elements of $P$. If $a<_Pb<_Pc<_Pd$ and $a,c\in I,$ then $b\in I$ or $d\in I$. 
 \end{lemma}

\begin{proof}
    Denote by $<_I$ the cyclic order on $I$ induced by $<_P$. If $a<_I i <_I c$ for some $i\in I$, we necessarily have $b\in I$, because $I$ is an interval. Similarly, if we have $c<_I i <_I a$ for some $i \in I$, we traverse the other direction in $P$ and necessarily contain $d$, as $c<_Id<_Ia$ from the cyclic order on $P$. Therefore we must have $b\in I$ or $d\in I$, as desired.
\end{proof}

\begin{corollary}[Lemma 4.8, \cite{PasqTZ2019}]\label{cor:ac-bord}
Let $I\in L_s$. If $a,c\in I$ and $b,d\notin I$ satisfy $a<b<c<d$ then $b\in \overline{a_s}$ or $d\in\overline{a_s}$. \end{corollary} 
\begin{proof}
It suffices to consider $I\in B_s$ with $a, c\in I$ and $b, d\not\in I$. By construction, we can realize all $I\in B_s$ as an interval in some $P_{s,h}$. Since we have $b,d \not \in I$, we cannot have both $b$ and $d$ in $P_{s,h}$ by Lemma~\ref{lemma: bordinI}. If we did, then either $b\in I$ or $d\in I$ would follow, a contradiction. Thus either $b$ or $d$ are in $P_s\backslash P_{s,h}$. From the definition of $P_{s, h}$, we can see that $P_s\backslash P_{s,h}\subseteq \overline {a_s}$. Thus, either $b$ or $d$ is in $\overline{a_s}$. 
\end{proof}
\begin{lemma}[Proposition 4.9, \cite{PasqTZ2019}]\label{lemma:LsLiSep} If $I\in L_i$ and $J\in L_j$ with $i\neq j$, then $I$ and $J$ are weakly separated. 
\end{lemma}
\begin{proof}
    Suppose that $I\in L_i$ and $J\in L_j$ are not weakly separated. Without loss of generality, take $i<j$. Then there exists $a,c\in I\backslash J$ and $b,d\in J\backslash I$ such that $a<b<c<d.$ By \Cref{cor:ac-bord}, either $b$ or $d$ is in $\overline{a_i}$. But by construction, no element in $J$ contains an element in $\overline{a_i}$ as it is from a previous equivalence class. This is a contradiction, so our assumption that $I$ and $J$ were not weakly separated must be false.
\end{proof}

\begin{lemma}\label{lemma:LsWeakOrbit} Given an interval $I\in B_s,$ the set $\{I, I+_n\ell, \dots, I +_{n} d\ell\}$ is weakly separated.
 \end{lemma}
\begin{proof}
By definition, $S_{P_s}$ respects the cyclic order on $P_s$. Then, by construction, for $x \in P_s$, $S_{P_s}^{\ell-s+1}(x)=x+\ell$. Since the elements generated from an element of $B_s$ are created precisely by adding $\ell$ modulo $n$ and since $S_{P_s}$ respects the ordering, each orbit generated by an element of $B_s$ is weakly separated.       
\end{proof}
\begin{lemma}\label{lemma:LsWeakNOrbit} The collection $L_s$ is weakly separated. 
\end{lemma}
\begin{proof}

By Lemma~\ref{lemma:LsWeakOrbit}, any two elements $I, J\in L_s$ belonging to the same $\rho^\ell$ orbit are weakly separated. Let $I, J \in L_s$ belong to distinct $\rho^\ell$ orbits and suppose that $I$ and $J$ are not weakly separated. Then, there must exist $a, c \in I\backslash J$ and $b,d \in J \backslash I$ such that $a < b < c < d$. By \Cref{cor:ac-bord}, we have that without loss of generality, $a, b \in \overline{a_s}$. 
Since $I\neq J$, we have that $|P_s|>k$, and we can uniquely give $J$ as $I(j, h)$ for some $j$. Consider the linear order $<_j$ induced by $j$. If $d<_j b$, then $[d, b]_{P_s}\subseteq J\cup \overline{a_s}$ by the definition of $P_{s,h}$. Moreover, $b\in J$ implies that $a\in J$ as well, because both $a$ and $b$ lie in $\overline{a_s}$ and if $b \leq a_s + h\ell$, then we also have $a\leq h\ell$. This contradicts the original assumption that $a\in I\backslash J$, so we must have $b<_j d$. As above, this implies that $[b, d]_{P_s}\subseteq J\cup \overline{a_s}$. Since $c\in [b,d]\backslash J,$ we have that $c\in \overline{a_s}$ in addition to $a$ and $b$.  Since $a,c$ are from the same interval, $I,$ $a<b<c$ is assumed. In a given orbit representative, if both $a$ and $c$ are in $\overline{a_s}$, that representative must contain all elements of $\overline{a_s}$ between $a$ and $c$, since 
$P_{s,h}$ is constructed by removing the last $d-h$ elements of $\overline{a_s}$ from $P_{s}$. Therefore, if $a,c\in I$ then $b\in I$ as well since we cannot skip elements in $\overline{a_s}$.  However $b\in J\backslash I$, providing us with our desired contradiction. Thus, $L_s$ is weakly separated. 

\end{proof}

\begin{corollary}\label{Cor:weakSeperation}
The collection $D=\bigcup L_s$ is weakly separated.
\end{corollary}
\begin{proof}\
    For a fixed $s$, every element in $L_s$ is weakly separated by \Cref{lemma:LsWeakNOrbit}. The union of all $L_s$ is then weakly separated by \Cref{lemma:LsLiSep}.
\end{proof}
\subsection{Generating a set when $n\neq d\ell$}\label{sub: not_divide} 

We now show how to generalize to the case where $1< \gcd(n, \ell) < \ell$ following the method laid out in \cite[Section 5]{PasqTZ2019}. Take $g=\gcd(n, \ell)$ and $d=\frac{n}{g}$. Let $C$ be the weakly separated collection generated by the algorithm from Section~\ref{sec: n=dl} with parameters $n'=d\ell, k'=k, \ell'=\ell$. We choose a total order such that the equivalence classes for $\overline{1}, \dots,\overline{g}$ are all greater than $ \overline{g+1},\dots, \overline{\ell}$. Note that $g$ is defined in terms of $n$ and not $n'$. We then remove the first $\ell-g$ sets from $C$ to obtain $k(n-k)+1$ elements and use them to recover a weakly separated collection $D\subseteq {[n]\choose k}$.

\begin{proposition}\label{Prop:genConstruct}
Let $D'=C \backslash \bigcup \limits_{s=1}^{\ell-g} L_s$. Then, $D'$ is a $\rho^\ell$-symmetric weakly separated collection with $k(n-k)+1$ elements. 
\end{proposition} 
\begin{proof}
$C$ is weakly separated by Corollary~\ref{Cor:weakSeperation}, and from Proposition~\ref{prop: cardinality_n=dl}, we also have that $|C|=k(d\ell-k)+1$. By \Cref{lemma: SizeBsSmalls}, when $1\leq s \leq \ell-g$, we have $|B_s|=k$. Moreover, by Lemma~\ref{lemma: sizeLsSmalls}, each element in $B_s$ will have an orbit of size $d,$ implying that $|L_s|=dk$. Each $L_s$ is disjoint, so we obtain $$\left|\bigcup \limits_{s=1}^{\ell-g} L_s\right|=(\ell-g)dk.$$
As such, the number of elements in $D'$ is given by:
\begin{align*}
    |D'|&=\left|C\backslash\bigcup \limits_{s=1}^{\ell-g} L_s\right|\\
    &=k(d\ell-k)+1-(\ell-g)dk\\
    &=k(n-k)+1
\end{align*}
Thus $D'$ is a symmetric, but not maximal, weakly separated collection with size $k(n-k)+1$. 
\end{proof}

We now turn towards defining a function to convert the $\rho^{\ell'}$-symmetric collection $D'\subseteq {[n'] \choose k'}$ to a $\rho^\ell$-symmetric collection $D\subseteq {[n]\choose k}$. We first give a sufficient condition for such a function to preserve weak separation.
\begin{lemma}\label{Lemma:incPreserve}
Let $F:[n]\to [m]$ be an injective increasing function. If $D\subseteq {[n]\choose k}$ is weakly separated, then $F(D)\subseteq {[m]\choose k}$ is weakly separated as well. 

\end{lemma} 
\begin{proof}
Given $n$ and $k$, let $D$ be a weakly separated collection and $F:[n] \rightarrow [m]$ be an injective increasing function, where $m > n$. Extend $F$ to a function on weakly separated collections simply by applying $F$ to each subset of $D$. Assume towards a contradiction that $F(D)$ is not weakly separated. 
Then, there must be elements $F(S)=\{F(a_1), F(a_2), \dots F(a_k)\}$ and $F(T)=\{F(b_1), F(b_2), \dots, F(b_k)\}$ in $F(D)$ such that, without the loss of generality, $F(a_i) < F(b_j) <F( a_{i'})< F(b_{j'})$. Since $F$ is increasing, it preserves order. Therefore, $a_i < b_j < a_{i'}<b_j'$, which contradicts the weak separation of $D$. Thus, $F$ preserves weak separation.       
\end{proof}

Now define $F:[n] \rightarrow [d\ell]$ by $F(a+gx)=a+\ell x$, for $1 \leq a \leq g$ and $0 \leq x < d$. Note that $F$ is injective and increasing, and that $d\ell > n$ when $\ell \nmid n$. If we restrict the codomain to elements in the equivalence classes $\overline{1},\dots,\overline{g}$, then $F$ becomes bijective, by construction. Denote the codomain of the restriction by $A$. 
\begin{lemma}[Lemma 5.2, \cite{PasqTZ2019}]\label{Lemma:Commutes}
The function $F$ satisfies $F \circ S_{[n]}=S_{A} \circ F$. 
 \end{lemma}
\begin{proof}
We will show $F$ commutes for a given element and obeys the cyclic ordering.  Recall that $F$ is increasing and a bijection. Furthermore, note that $g>1$ as $n$ and $\ell$ cannot be coprime. As such, $F(1)$ is uniquely written $F(1+0g)$ and thus $x$ is $0$. Then $F(1)=1$.
We first would like to show that $S_A(F(x))=F(x+1)$. Suppose not. Then $S_A(F(x))=F(x+i)$ for $i\neq 1$. Since $F$ is strictly increasing, no later element can map to $F(x+1)$. As such, $F(x+1)$ is not in the image of $F.$ However $F$ is bijective, so this is a contradiction. Thus $S_A(F(x))=F(x+1)=F(S_{n}(x)).$ 
The following computation verifies that $F$ obeys the cyclic ordering: $$F(S_{[n]}(n))=F(n+1)=S_{A}(n)=S_A(g+\ell(d-1))=S_A(F(g+g(d-1)))=S_A(F(n)).$$
\end{proof}

\begin{lemma}\label{lemma:set-symmetry}
 The set $D=F^{-1}(D')$ is $\rho^\ell$-symmetric.
\end{lemma}
\begin{proof}
By \Cref{Lemma:Commutes}, for $I \in D$, $F(I)$ is fixed under addition by $\ell$ modulo $d\ell$. Therefore, if $F^{-1}(F(I))=I \in D$, then $F^{-1}(F(I)+\ell) \in D$ as well. Note that $F^{-1}(\ell)=g$. Thus, $F^{-1}(F(I)+\ell)=I+g$. Since $g \mid \ell$, $I+\ell$ is also contained in $D$.   
\end{proof}

We now prove Theorem~\ref{thm: intro_WSCs}, restated here for clarity.
\begin{theorem}
$k \equiv c \pmod d$ where $c \in \{-1,0, 1\}$ if and only if $D$ is a $\rho^\ell$-symmetric maximal weakly separated collection.
\end{theorem}
\begin{proof}
From \Cref{prop: necessary}, we know that $k\equiv c \pmod d$ is a necessary condition for the existence of a $\rho^\ell$-symmetric maximal weakly separated collection. Thus, it suffices to prove that the algorithm described above generates a $\rho^\ell$-symmetric weakly separated collection of size $k(n-k)+1$ in this case. If $n=d\ell$, then this follows from Proposition~\ref{prop: cardinality_n=dl} and Corollary~\ref{Cor:weakSeperation}. If $n\neq d\ell$, then \Cref{Prop:genConstruct}, $D'$ implies that is a weakly separated collection with $k(n-k)+1$ elements. By \Cref{Lemma:incPreserve}, $D$ is weakly separated, and since $F$ is a bijection, $|D|=k(n-k)+1$, implying that $D$ is maximal. Lastly, by \Cref{lemma:set-symmetry}, $D$ is $\rho^\ell$-symmetric. Thus, $D$ is a $\rho^\ell$-symmetric maximal weakly separated collection. 
\end{proof}

\bibliographystyle{alpha}
\bibliography{Biblio}

\begin{thebibliography}{CLSBW23}

\bibitem[ABL22]{ABL22}
Byung~Hee An, Youngjin Bae, and Eunjeong Lee.
\newblock {L}agrangian fillings for {L}egendrian links of finite or affine {D}ynkin type, 2022.
\newblock arxiv:2201.00208.

\bibitem[Ad90]{ArnoldSing}
V.~I. Arnol\textquotesingle~d.
\newblock {\em Singularities of caustics and wave fronts}, volume~62 of {\em Mathematics and its Applications (Soviet Series)}.
\newblock Kluwer Academic Publishers Group, Dordrecht, 1990.

\bibitem[Cas21]{CasalsLagSkel}
Roger Casals.
\newblock {L}agrangian skeleta and plane curve singularities.
\newblock {\em JFPTA}, Viterbo 60, 2021.

\bibitem[CG22]{CasalsGao}
Roger Casals and Honghao Gao.
\newblock Infinitely many {L}agrangian fillings.
\newblock {\em Ann. of Math. (2)}, 195(1):207--249, 2022.

\bibitem[CGW25]{PattonGithub}
Vincent Chen, Patton Galloway, and Luciana Wei.
\newblock Symmetric weakly separated collection algorithm.
\newblock URL: \url{https://github.com/pattonga/Symmetric-Weakly-Separated-Collection-Generator}, August 2025.

\bibitem[CLSBW23]{CLSW2023}
Roger Casals, Ian Le, Melissa Sherman-Bennett, and Daping Weng.
\newblock Demazure weaves for reduced plabic graphs (with a proof that muller-speyer twist is donaldson-thomas), 2023.

\bibitem[CN22]{CasalsNg}
Roger Casals and Lenhard Ng.
\newblock Braid loops with infinite monodromy on the {L}egendrian contact {DGA}.
\newblock {\em J. Topol.}, 15(4):1927--2016, 2022.

\bibitem[CS14]{ChekhovShapiroGeneralizedCluster}
Leonid Chekhov and Michael Shapiro.
\newblock Teichm\"uller spaces of {R}iemann surfaces with orbifold points of arbitrary order and cluster variables.
\newblock {\em Int. Math. Res. Not. IMRN}, (10):2746--2772, 2014.

\bibitem[CSC24]{capovillasearle2023newton}
Orsola Capovilla-Searle and Roger Casals.
\newblock On {N}ewton polytopes of {L}agrangian augmentations.
\newblock {\em Bull. Lond. Math. Soc.}, 56(4):1263--1290, 2024.

\bibitem[CW24]{CasalsWeng}
Roger Casals and Daping Weng.
\newblock Microlocal theory of {L}egendrian links and cluster algebras.
\newblock {\em Geom. Topol.}, 28(2):901--1000, 2024.

\bibitem[CZ22]{CZ2022}
Roger Casals and Eric Zaslow.
\newblock Legendrian weaves: {$N$}-graph calculus, flag moduli and applications.
\newblock {\em Geom. Topol.}, 26(8):3589--3745, 2022.

\bibitem[DRG21]{DRG21}
Georgios Dimitroglou~Rizell and Roman Golovko.
\newblock On {L}egendrian products and twist spuns.
\newblock {\em Algebr. Geom. Topol.}, 21(2):665--695, 2021.

\bibitem[EHK16]{EHK}
Tobias Ekholm, Ko~Honda, and Tam\'{a}s K\'{a}lm\'{a}n.
\newblock {L}egendrian knots and exact {L}agrangian cobordisms.
\newblock {\em J. Eur. Math. Soc. (JEMS)}, 18(11):2627--2689, 2016.

\bibitem[EK08]{EK}
Tobias Ekholm and Tam\'{a}s K\'{a}lm\'{a}n.
\newblock Isotopies of {L}egendrian 1-knots and {L}egendrian 2-tori.
\newblock {\em J. Symplectic Geom.}, 6(4):407--460, 2008.

\bibitem[EP96]{EliashbergPolterovich96}
Y.~Eliashberg and L.~Polterovich.
\newblock Local {L}agrangian {$2$}-knots are trivial.
\newblock {\em Ann. of Math. (2)}, 144(1):61--76, 1996.

\bibitem[Fra20a]{fraser2018braid}
Chris Fraser.
\newblock Braid group symmetries of {G}rassmannian cluster algebras.
\newblock {\em Selecta Math. (N.S.)}, 26(2):Paper No. 17, 51, 2020.

\bibitem[Fra20b]{Fraser2020}
Chris Fraser.
\newblock Cyclic symmetry loci in {G}rasssmannians.
\newblock arXiv:2010.05972, 2020.

\bibitem[FWZ20]{FWZ1}
Sergey Fomin, Lauren Williams, and Andrei Zelevinsky.
\newblock Introduction to cluster algebras: Chapters 1-3.
\newblock arXiv:1608.05735, 2020.

\bibitem[FWZ25]{Fomin7}
Sergey Fomin, Lauren Williams, and Andrei Zelevinsky.
\newblock Introduction to cluster algebras. chapter 7.
\newblock arXiv:2106.02160, 2025.

\bibitem[Gal25]{PlabicTilingCalc}
Pavel Galashin.
\newblock Plabic tilings.
\newblock URL: \url{https://www.math.ucla.edu/~galashin/plabic.html}, August 2025.

\bibitem[Gei08]{Geiges08}
Hansj{\"o}rg Geiges.
\newblock {\em An introduction to contact topology}, volume 109 of {\em Cambridge Studies in Advanced Mathematics}.
\newblock Cambridge University Press, Cambridge, 2008.

\bibitem[GSW24]{GSW}
Honghao Gao, Linhui Shen, and Daping Weng.
\newblock Augmentations, fillings, and clusters.
\newblock {\em Geom. Funct. Anal.}, 34(3):798--867, 2024.

\bibitem[HR25]{HughesRoy}
James Hughes and Agniva Roy.
\newblock Legendrian doubles, twist spuns, and clusters.
\newblock arXiv:2505.17901, 2025.

\bibitem[Hug23]{Hughes2021}
James Hughes.
\newblock Weave-realizability for {D}--type.
\newblock {\em Algebr. Geom. Topol.}, 23(6):2735--2776, 2023.

\bibitem[Hug24]{Hughes2024}
James Hughes.
\newblock Legendrian loops and cluster modular groups.
\newblock https://arxiv.org/abs/2403.12951, 2024.

\bibitem[K{\'a}l05]{Kalman2005}
Tam{\'a}s K{\'a}lm\'an.
\newblock Contact homology and one parameter families of {L}egendrian knots.
\newblock {\em Geom. Topol.}, 9:2013--2078, 2005.

\bibitem[Kau24]{KaufmanSpecialFolding}
Dani Kaufman.
\newblock Special folding of quivers and cluster algebras.
\newblock {\em Math. Scand.}, 130(2):237--256, 2024.

\bibitem[OPS15]{OPS2015}
Suho Oh, Alexander Postnikov, and David~E. Speyer.
\newblock Weak separation and plabic graphs.
\newblock {\em Proceedings of the London Mathematical Society}, 110(3):721--754, 2015.

\bibitem[Pan17]{YuPan}
Yu~Pan.
\newblock Exact {L}agrangian fillings of {L}egendrian {$(2,n)$} torus links.
\newblock {\em Pacific J. Math.}, 289(2):417--441, 2017.

\bibitem[Pos06]{postnikov2006}
Alexander Postnikov.
\newblock Total positivity, {G}rassmannians, and networks, 2006.

\bibitem[PTZ20]{PasqTZ2019}
Andrea Pasquali, Erik Th\"ornblad, and Jakob Zimmermann.
\newblock Existence of symmetric maximal noncrossing collections of {$k$}-element sets.
\newblock {\em J. Algebraic Combin.}, 52(1):41--58, 2020.

\bibitem[Sco06]{Scott06}
Joshua~S. Scott.
\newblock Grassmannians and cluster algebras.
\newblock {\em Proc. London Math. Soc. (3)}, 92(2):345--380, 2006.

\bibitem[STWZ19]{STWZ}
Vivek Shende, David Treumann, Harold Williams, and Eric Zaslow.
\newblock Cluster varieties from {L}egendrian knots.
\newblock {\em Duke Math. J.}, 168(15):2801--2871, 2019.

\bibitem[Tho25]{Thomson25}
Bryce Thomson.
\newblock The {L}egendrian {H}opf link has exactly two {L}agrangian fillings.
\newblock arXiv:2506.15111, 2025.

\bibitem[TZ18]{TreumannZaslow}
David Treumann and Eric Zaslow.
\newblock Cubic planar graphs and {L}egendrian surface theory.
\newblock {\em Adv. Theor. Math. Phys.}, 22(5):1289--1345, 2018.

\end{thebibliography}

\end{document}